\documentclass[10pt]{amsart}

\usepackage{amsfonts,amssymb,amscd,amstext}

\usepackage[paper size={210mm,297mm},left=34mm,right=34mm,top=43.5mm,bottom=43.5mm]{geometry}
\usepackage[charter]{mathdesign}
\usepackage[colorlinks=true,linkcolor=blue,citecolor=red]{hyperref}
\usepackage{graphicx}
\usepackage[utf8]{inputenc}
\usepackage{esint}
\usepackage{enumerate}
\usepackage{verbatim}
\usepackage[dvips]{color}
\usepackage{amssymb}
\usepackage[active]{srcltx}

\usepackage{latexsym}
\usepackage{amsmath}
\usepackage{float}
\usepackage{mathrsfs}

\renewcommand{\leq}{\leqslant}
\renewcommand{\geq}{\geqslant}

\newcommand{\rr}{{\mathbb{R}}}

\newcommand{\hh}{{\mathbb{H}}}

\newcommand{\hhh}{\mathcal{H}}

\newcommand{\escpr}[1]{\big<#1\big>}
\newcommand{\Sg}{\Sigma}

\newcommand{\Om}{\Omega}
\newcommand{\eps}{\varepsilon}

\newcommand{\area}{A}

\newcommand{\nau}{\nabla^u}
\newcommand{\na}[1]{\nabla^{#1}}
\newcommand{\mc}[1]{\mathcal{#1}}
\newcommand{\somma}[1]{\sum\limits_{#1=1}\limits^{2n-1}}

\usepackage{color}
\definecolor{grey}{rgb}{.7,.7,.7}

\DeclareMathOperator{\divv}{div}

\newtheorem{theorem}{Theorem}[section]
\newtheorem{proposition}[theorem]{Proposition}
\newtheorem{lemma}[theorem]{Lemma}
\newtheorem{corollary}[theorem]{Corollary}

\theoremstyle{definition}

\newtheorem{definition}{Definition} 

\theoremstyle{remark}

\numberwithin{equation}{section}

\begin{document}

\title[stable minimal graphs in $\hh^n$]{Stable minimal graphs in the Heisenberg group $\hh^n$}

\author[G.~Citti]{Giovanna Citti} \address{Departimento di
Matematica \\
Universit\`a di Bologna \\ Piazza di porta san Donato 5, Bologna \\ Italia}
\email{giovanna.citti@unibo.it}

\author[M.~Galli]{Matteo Galli} \address{Departimento di
Matematica \\
Universit\`a di Bologna \\ Piazza di porta san Donato 5, Bologna \\ Italia}
\email{matteo.galli8@unibo.it}

\date{\today}

\thanks{ The research leading to these results has received funding from the People
Programme (Marie Curie Actions) of the European Union's Seventh
Framework Programme FP7/2007-2013/ under REA grant agreement n. 607643}

\begin{abstract}
We prove that a strictly stable minimal $C^2_h$ intrinsic graph G is locally area-minimizing, i.e. given any  $C^1_h$ graph $S$ with the same boundary, $\text{Area}(G)<\text{Area}(S)$ unless $G=S$. As  a consequence we show the existence and the uniqueness of $C^\infty$ minimal graphs with prescribed small boundary datum.
\end{abstract}

\maketitle

\thispagestyle{empty}

\bibliographystyle{amsplain}% \nocite{*}

\tableofcontents

\section{Introduction}

The Heisenberg group, denoted as $\hh^n$ can be identified with 
$\rr^{2n+1}$, with the choice at every point of an 
\emph{horizontal $2n$-dimensional distribution} $\hhh$, such that 
$[\hhh,\hhh]=\hhh_2 $ has dimension 1, 
\begin{equation}\label{strat} 
T\hh^n = \hhh\oplus \hhh_2, 
[\hhh,\hhh]=\hhh_2 \text{ and  }[\hhh, \hhh_0]=0.\end{equation} 

Starting from the seminal paper \cite{Ga-Nh} by Garofalo and Nhieu, where general properties and existence of sets with minimum perimeter are proved in Carnot groups, a remarkable effort has been devoted to the development of an exhaustive theory for minimizers of the area functional in the sub-Riemannian setting. However the theory is still very far from being complete. 

Due to the lack of symmetry of the space, 
surfaces can have very different expression if expressed as graph 
on different variables.  
Identifying via the exponential coordinates the 
Heisenberg group with its algebra, we can consider graphs defined on the horizontal distribution 
$\hhh$ with values in $\hhh_2$. 
Graphs of this type with prescribed mean curvature $f$, called $t$-graphs, are critical points of the functional 
\begin{equation*}
\label{eq:F}
\mathcal{F}(u)=\int_\Om |\nabla u+\vec{F}|+\int_\Om fu,
\end{equation*}
on a domain $\Om$, where $\vec{F}$ is a vector field and $f\in L^\infty(\Om)$. In \cite{MR2262784} Cheng et al. proved the existence of $C^1$ minimizing $t$-graphs in $\hh^n$ that 
at our knowledge, this is the only existence results for prescribed mean curvature graphs. 

\smallskip

On the other side while choosing graphs whose normal at every point belongs
to the horizontal distribution we obtain the so called intrinsic graphs 
introduced by Franchi, Serapioni and Serra Cassano in \cite{FSSC4} and \cite{FSSC1}. It has been proved, 
in \cite{MR2223801} and \cite{MR2263950}, that in this case there are 
non-linear vector fields $$(X_{1,u}, \cdots, X_{n,u})$$ such that the area functional of the 
graph of $u$ can be written as 
\[
\area(u)=\int_\Omega  \sqrt{1+|\na{u} u|^2}\,\, d\mc{L}^{2n}. 
\]

We remark that this sub-Riemannian area functional is not convex. 
This is why standard existence technique do not apply. However existence of minima of this functional have been obtained
by Serra Cassano and Vittone, \cite{MR3276118}, but they have only $BV$ regularity. 
In particular it is not clear if they satisfy the associated first variation equation 
which express the curuvature of a surface: 
\begin{equation}\label{eq:minimalintro}
H(u)=d\area(u) =\somma{i} X_i^{u}\bigg(\frac{X_i^{u} (u)}{ \sqrt{1+|\na{u} (u)|^2}}\bigg)=0.
\end{equation} 
This is a second order quasilinear subelliptic equation, but due to the 
expression of the vector fields $X_{i,u}$, 
the regularity of the solution 
is not understood, and there is still a big gap between 
this existence result for minima and their regularity properties. 
Regularity results in the three-dimensional case, with special attention to the first Heisenberg group $\hh^1$ has been obtained only for Lipschitz continous solutions.  Among them, we would like to stress \cite{CHY2}, 
%Cheng, Jih-Hsin and Hwang, Jenn-Fang and Yang, Paul, Regularity of {$C^1$} smooth surfaces with prescribed {$p$}-mean curvature in the {H}eisenberg group
\cite{MR2583494} 
%Capogna, Luca and Citti, Giovanna and Manfredini, Maria
, \cite{MR2831583}, %Barbieri, D. and Citti, G.
%Very recently, in 
\cite{MR3406514}, \cite{MR3412382} and \cite{MR3474402}. 
In dimension bigger then three the only regularity result in $\hh^n$, $n>1$, is proved by Capogna et al. in \cite{MR2774306}, 
always for Lipschitz continuous solutions. 
The main obstacle to study directly the the problem
\begin{equation}\label{critical_points}
\begin{cases}
H_v =0, & \text{in } \Omega\\
v=\phi, & \text{in } \partial\Omega
\end{cases}
\end{equation}
is the fact that the operator 
\begin{equation}\label{Luvu}L_u(v) = dH(u)(v)\end{equation} does not satisfy the maximum principle for general functions $u$. 
Indeed in the subriemannian setting the stability operator  $L_u v$ takes the form
\[
\begin{aligned}
L_u v &=\somma{i,j}X_i^{u} (a_{ij}(\na{u}u)\,X_j^u v)+\somma{i} (\partial_{2n}u) \, a_{i 2n-1}(\na{u}u) \, X_i^u v \\
&+ \bigg(  \partial_{2n}\bigg( \frac{ X^u_{2n-1}u }{\sqrt{1+|\na{u}u|^2}} \bigg) +\somma{i} X_i^{u}(\partial_{2n}u \, a_{i 2n-1}(\na{u}u)) \bigg)\, v,
\end{aligned}
\]
(see Proposition~\ref{prop:2ndvariation}), and nothing is known on the sign of the 
zero order term.  In addition Schauder estimates are not available for 
general subriemannian vector fields.

\smallskip

In this paper we investigate properties of strictly stable critical points, 
which are graphs $G_u$ with vanishing mean curvature in an open set $\Omega$ 
such that, for all compact test functions $v \in C^2_h(\Omega)$, $v\not\equiv 0$, the \emph{index form}
\[
\mathcal{I}(v,v)=-\,\escpr{L_u(v), v}
\]
is strictly positive, where $L_u(v)$ is defined in \eqref{Luvu}. 
Our main result is stated in Theorem \ref{thm:calibration} below: 
\begin{quotation}
Let $u\in C^2_h(\tilde{\Om})$ and let $u$ be a strictly stable critical point of the 
area functional in a domain $\Omega\subset\tilde{\Om}$. Then there exists a tubular neighborhood $U$ of $\Omega$ such that for any  $C^1_h$ graph $S\subset U$, $\partial\Omega=\partial S$, we have $\area(G_u)< \area(S)$ or $G_u=S$.
\end{quotation}

This result extends to the present setting the one proved in Riemannian manifolds by White \cite{MR1305283}, using regularity theorems from geometric measure theory.   Grosse-Brauckmann \cite{MR1432843}, see also \cite[\S 109]{MR1015936}, show that another way to prove this result is to foliate a neighbourhood of a strictly stable extremal with stationary surfaces. The field of normals to the leaves is a calibration, and the statement follows by the divergence theorem. 

When $u$ is 
a strictly stable critical point $u$, $L_u(v)$ is positive, 
hence invertible. Hence it is possible to show 
an existence result in a neighborhood of any stable point. 
In particular, since $u=0$ has this property we deduce that 

\begin{quotation}\label{intro_existence}
There exists $\eps_0>0$ such that if 
$||\phi ||_{L^\infty(\partial \Omega)} \leq \eps_0$, 
then the problem \eqref{critical_points} 
has an unique solution $v\in C^\infty(\Omega)$. \end{quotation}

\smallskip

The delicate aspect of the proof of the existence theorem is the lack of Schauder 
estimates at the boundary. As we mentioned before, they  are not known in the 
subriemannian setting and the only 
result in this direction is due to Jerison \cite{MR639800}. However his proof 
can not be repeated in general Lie groups, since 
it is based on Fourier transform. On the contrary,  
internal Schauder estimates are well known in this context,
after the results of \cite{MR0436223} (we also quote a more recent contribution 
\cite{MR2298970}). Hence we prove an ad hoc version of Schauder estimates 
with a penalization on the boundary of the set, 
which are sufficient to obtain the result. 

\smallskip

Using this existence result, we are 
able to follow the same idea of  \cite{MR1432843}  
and construct a foliation by minimal graphs in a tubular neighborhood of a strictly stable minimal graph. The positivity of the  operator $L_u$  for strictly stable function $u$ implies that 
the operator $L_u$ 
satisfies the maximum principle (see also Proposition \ref{lem:maximum principle}). As a consequence 
we will establish the proof of Theorem \ref{thm:calibration}.

\smallskip

As a corollary we will deduce 
that this critical point of the area functional, found above for 
$||\phi ||_{L^\infty(\partial \Omega)} \leq \eps_0$, 
is indeed a stable minimum with prescribed boundary datum. 

\smallskip

The paper is organized as follows. In Section 2 we provide the necessary background on the sub-Riemannian Heisenberg group and intrinsic graphs with prescribed mean curvature. In Section 3 we introduce the stability operator. In section 3 we prove the Schauder estimates and our existence result. The local area-minimizing property of a strictly stable minimal intrinsic graph will appear in Section 5.

\section{Preliminaries}
In this section we gather some results to be used in later sections.

\subsection{The Heisenberg group $\hh^n$}
The structure of the Heisenberg group $\hh^n$ can be modeled on $\rr^{2n+1}$ using the following basis of left-invariant vector fields 
\[
\begin{split}
&X_i=\frac{\partial}{\partial x_i}, i=1,\dots, n-1, \quad  
X_i=\frac{\partial}{\partial x_i}-x_{i-n+1}\frac{\partial}{\partial x_{2n}}, i=n,\dots, 2n-2,
\\ &X_z=\frac{\partial}{\partial z}, \quad X_{2n-1}=\frac{\partial}{\partial x_{2n-1}}-z\frac{\partial}{\partial x_{2n}}, \quad
X_{2n}=\frac{\partial}{\partial x_{2n}}.
\end{split}
\]
The vector fields $\{X_z,X_1,\dots, X_{2n-1}\}$ generate the \emph{horizontal distribution} $\hhh$, while  $T$ is called the Reeb vector field and it is transverse to $\hhh$. A vector  field $X$ is called \emph{horizontal} if $X\in \hhh$. A horizontal curve is a $C^1$ curve whose tangent vector lies in the horizontal distribution. 

Note that 
\[
[X_z,X_{2n-1}]=[X_i,X_{i+n-1}]=T, \quad i=1,\dots, n-1,
\] 
while the other commutators vanish, so that $\hhh$ is a bracket-generating distribution and $\hh^n$ has vanishing pseudo-hermitian Webster curvature and pseudo-hermitian torsion, see \cite{Dr-To} or \cite{Gaphd}. For this reason the Heisenberg group is the model example of pseudo-hermitian manifolds and play the same role that the Euclidean space has with respect to a Riemannian manifold.

\subsection{The left invariant metric}
\label{subsec:g}
We shall consider on $\hh^n$ the Riemannian metric $g=\escpr{\cdot\,,\cdot}$ so that $\{X_z,X_1,\dots, X_{2n-1},T\}$ is an orthonormal basis at every point.  The restriction of $g$ to $\mathcal{H}$ coincides with the usual sub-Riemannian metric in $\hh^n$ induced by the vector fields $X_z,X_1,\dots, X_{2n-1}$. %The orthogonal \emph{horizontal projection} of a tangent vector $U$ onto $\mathcal{H}$ will be denoted by $U_{h}$.  A vector field $U$ is \emph{horizontal} if and only if $U=U_h$.

For any tangent vector $U$ on $\hh^n$ we define $J(U)=D_UT$, where $D$ is the Levi-Civita connection associated to the Riemannian metric $g$.  Then we have 
\[
J(X_z)=X_{2n-1},\, J(X_{2n-1})=-X_z,\, J(X_i)=X_{i+n-1},\, J(X_{i+n-1})=-X_i,
\]
for $i=1,\dots, n-1$, and $J(T)=0$, so that $J^2=-\text{Id}$ when restricted to $\mathcal{H}$.  The involution $J:\mathcal{H}\to\mathcal{H}$ provides a complex structure on $\hh^n$, see \cite{Bl}.

\subsection{Geometry of surfaces in $\hh^n$}
Given a ${C}^1$ surface $\Sigma$ immersed in $\hh^n$ we define the \emph{sub-Riemannian area} of $\Sg$ by
\begin{equation}\label{eq:area2}
A(\Sigma)=\int\limits_{\Sigma} |N_h| d\Sigma,
\end{equation}
where $N$ is the unit normal vector with respect to the metric $g$, $N_h$ is the orthogonal projection of $N$ to $\mathcal{H}$  and $d\Sigma$ is the Riemannian area element of $\Sigma$. The singular set $\Sigma_0$ consists of those points $p$ where $\hhh_p$ coincides with the tangent plane $T_p\Sg$ of $\Sg$.  We define the \emph{horizontal unit normal vector} $\nu_h(p)$ and the \emph{characteristic vector field} $Z(p)$ by
\begin{equation}\label{def:Znu_h}
\nu_h(p):=\frac{N_h(p)}{|N_h(p)|}, \quad Z(p):=J(\nu_h)(p)
\end{equation}
for all $p\in\Sg-\Sg_0$. Since $Z_p$ is orthogonal to $\nu_h$ and horizontal, we get that $Z_p$ is tangent to $\Sg$.

\subsection{Euclidean Lipschitz graphs in $\hh^n$}

Let $W=\{(0,x)\in\hh^n:x\in \rr^{2n}\}$,  we consider the graph $\Sg=\{(z,x): z=u(x), x\in \Omega\}\subset \hh^n$ of an Euclidean Lipschitz function $u:\Omega\subset W\rightarrow \rr$. Let $\{E_i\}_{i=1,\dots, 2n-1}$ be a basis of the horizontal tangent space $T\Sg\cap\hhh$, where
\[
\begin{split}
&E_i= (X_i u) \,X_z+X_i ,\quad   i=1,\dots, 2n-1.
%&X_{2n-1,u}= \frac{\partial}{\partial x_{2n-1}}-u(x)\frac{\partial}{\partial x_{2n}}.
\end{split}
\]
We  denote by $\{X_{i}^u\}_{i=1,\dots, 2n-1}$ a basis of the horizontal tangent vectors  projected to $\Om$
\[
\begin{split}
&X_{i}^u=X_i ,\quad   i=1,\dots, 2n-2 \\
&X_{2n-1}^u= \frac{\partial}{\partial x_{2n-1}}+u(x)\frac{\partial}{\partial x_{2n}}.
\end{split}
\]
The gradient $\nabla^u$ is defined as
\[
\nabla^u=(X_{1}^u, \dots, X_{2n-1}^u)
\]
and $\nabla^u u$ is well-defined and continuous in $\Omega$, since $G_u$ in a so-called \emph{intrinsic graph}, \cite{FSSC1} and \cite{MR2600502}, in fact the integral curves of $X_z$ starting from $\Omega$ meet $G_u$ in exactly one point. The area formula \eqref{eq:area2} for a graph $G_u$ can be expressed as
\begin{equation}\label{eq:areagraph}
A(G_u)=\int_\Om (1+|\nabla^u u|^2)^{1/2}d\mathcal{L}^{2n}, 
\end{equation}
where $\mathcal{L}^{2n}$ denotes the Lebesgue measure on $\Omega$, \cite[Proposition~2.22]{MR2223801}.

\subsection{Graphs with prescribed mean curvature}

Let $G_u$ the graph of an Euclidean Lipschitz function $u:\Omega\subset W\rightarrow \rr$. $G_u$  has \emph{prescribed mean curvature} $f$ if it is
a critical point of the functional
\begin{equation}\label{eq:prescribedfunctional}
A(G_u\cap B) -\int_{E_u\cap B} f,
\end{equation}
for any bounded open set $B$ in the cylinder $\{(z,x)\in \hh^n: x\in \Omega\}$. Here we have denoted by $E_u=\{(z,x)\in \hh^n: x\in \Omega, z<u(x)\}$ the subgraph of $G_u$. We remark that our definition is the counterpart of the ones given in three-dimensional sub-Riemannian contact manifolds, \cite{MR3412382} and \cite{MR3474402}, and in the Euclidean setting, \cite[(12.32) and Remark~17.11]{MR2976521}.

\subsection{Distance generated by vector fields}

Note that the vector fields $X_{1}^u, \dots, X_{2n-1}^u$ satisfy Hormander's finite rank condition in $\rr^{2n}$. Consequently they give rise to a control distance $d_u$, whose metric balls $B_r$ of radius $r$ have volume comparable to $r^{2n+1}$, where $2n+1$ is the homogeneous dimension of the space $(\rr^{2n},d_u)$. The distance $d_u$ coincides with the sub-Riemannian distance restricted to the vertical plane containing $\Omega$, see \cite{NSW}.

\subsection{Function spaces} Let $f$ be a continuous function on $E\subset \hh^n$. We say that $f\in C^k_h(E)$ if $Y_1(\dots(Y_k(f)))$ exists continuous, for any $Y_1, \dots, Y_k\in \hhh$. It is easy to check that $f\in C^{2k}_h(E)$ implies $f\in C^k(E)$. 

Given $0<\alpha<1$, we define $C^{0,\alpha}_d(\Omega)$ as the space composed by all functions $f:\Omega\rightarrow \rr$ having finite the following norm
\[
||f||_{\alpha, d, \Omega}:= \sup\limits_\Omega |f|+\sup\limits_{x,y\in \Omega, x\neq y} d^\alpha_{x,y} \frac{|f(x)-f(y)|}{d_u(x,y)|^\alpha},
\]
where 
\[
d_{x,y}:=\min(d_u(x,\partial\Omega), d_u(y,\partial\Omega))
\]
and $d_u$ denotes the distance induced by the vector fields $X^u_{i}$, $i=1, \dots, 2n-1$.
We also say that $f\in C^{2+\alpha}_d(\Omega)$ if the norm
\[
||f||_{2+\alpha,d,\Omega}:=  \sup\limits_\Omega |f|+ \sum_{i=1}^{2n-1}\sup\limits_\Omega d |X_i^u f|+\sum_{i,j=1}^{2n-1}||d^2 X_i^uX_j^u f||_{\alpha,d,\Omega}<+\infty.
\]

\subsection{First variation of the Area functional}
We first recall that, if $u$ is a graph, Lipschitz with respect to the 
Euclidean metric, the area of the graph $G_{u}$ is
\[
\area(G_{u})=\int_\Omega  \sqrt{1+|\na{u} u|^2}\,\, d\mc{L}^{2n}
\]
(see for example \cite{MR2979606}, \cite{DGNAV}, \cite{MR2774306}).  
Moreover, if $u$ is a critical point of the functional \eqref{eq:prescribedfunctional} we have that 
\begin{equation}\label{eq:1stvariationdebole}
\int\limits_\Omega \bigg\{   \frac{ \sum\limits_{i=1}\limits^{2n-2} X_iu\,X_i\phi + 
X_{2n-1}^uu(X_{2n-1}^u)^*\phi}{\sqrt{1+|\nau u|^2}} - f \phi \bigg\}\, d\mathcal{L}^{2n}=0
\end{equation}
for any $\phi\in C^1_0(\Omega)$. %{\color{red} (I think we can consider the same expression without 
In order to simplify notations we will denote 
\begin{equation}\label{divergence_term}A_i(\nau u) = \frac{ X_i u}{\sqrt{1+|\nau u|^2}} \, \text{ for } i= 1\cdots 2n-2\quad 
A_{2n-1}(\nau u) = \frac{ X^u_{2n-1} u}{\sqrt{1+|\nau u|^2}}.\end{equation}

As a consequence the definition of \emph{mean curvature} for the graph $G_{u}$ reads (see for example \cite{MR2979606}, \cite{DGNAV}, \cite{MR2774306}): 
\begin{equation}\label{def:H_u}
H_{u}:=\somma{i} X_i^{u}\bigg(\frac{X_i^{u} u}{ \sqrt{1+|\na{u} u|^2}}\bigg) = 
\somma{i} X_i^{u}A_i(\nau u) 
\end{equation}
We can also consider the operator of the minimal surface equation in the non-divergence form 
\begin{equation}\label{def:M_u}
M_{u}=\somma{i} a_{ij}(\nabla^u u)X_i^{u}X_j^{u} (u),
\end{equation}
where 
\begin{equation}\label{def:aij}
a_{ij}:\rr^{2n}\rightarrow \rr, \quad  a_{ij}(p)=\delta_{ij}-\frac{p_i p_j}{1+|p|^2}. 
\end{equation}

\section{Stability operator for minimal graphs}

In this section we define the stability operator for subriemannian minimal graphs.

\begin{proposition}\label{prop:2ndvariation} Let $G_u=\{z=u(x): x\in \Omega \}$ be a $C^2_h$  graph in $\hh^n$ and let $v\in C^{2}_h(\Omega)$. We consider a variation of $G_u$ of the form 
$G_{u+sv}=\{z=u(x)+sv(x): x\in \Omega \}$. Then the second variation of the area of $G_u$ is
\[
\frac{d^2}{ds^2}\bigg|_{s=0} \area(G_{u+sv})=\int_\Omega v\, L_u v \, d\mathcal{L}^{2n},
\]
where
\begin{equation}\label{def:Lu}
\begin{aligned}
L_u v &=\somma{i,j}X_i^{u} (\frac{a_{ij}(\na{u}u)}{\sqrt{1+|\na{u}u|^2}}\,X_j^u v)+\somma{i} (\partial_{2n}u) \, a_{i 2n-1}(\na{u}u) \, X_i^u v \\
&+ \bigg(  \partial_{2n}\bigg( \frac{ X^u_{2n-1}u }{\sqrt{1+|\na{u}u|^2}} \bigg) +\somma{i} X_i^{u}(\partial_{2n}u \, a_{i 2n-1}(\na{u}u)) \bigg)\, v,
\end{aligned}
\end{equation}
and the functions $a_{ij}$ are defined in \eqref{def:aij}.

\end{proposition}

\begin{proof} It is a standard computation show that 
\begin{equation}\label{eq:dHusv}
\begin{aligned}\nonumber
\frac{d}{ds}H_{u+sv} =& \somma{i} X_i^{u+sv}\bigg(\frac{X_i^{u+sv} (v)}{ \sqrt{1+|\na{u+sv} (u+sv)|^2}}\bigg)-\somma{i,j}X_i^{u+sv}\bigg(\frac{X_i^{u+sv} (u+sv)  X_j^{u+sv} (u+sv) X_j v }{ (1+|\na{u+sv} (u+sv)|^2)^{3/2}}\bigg)\\
&+v\partial_{2n}\bigg(\frac{X_{2n-1}^{u} (u+sv)}{ \sqrt{1+|\na{u+sv} (u+sv)|^2}}\bigg)+   X_{2n-1}^{u}\bigg(\frac{ v\partial_{2n}(u+sv)}{ \sqrt{1+|\na{u+sv} (u+sv)|^2}}\bigg)\\
& - \somma{i}  X_i^{u+sv} \bigg(\frac{X_i^{u+sv} (u+sv)  X_{2n-1}^{u+sv} (u+sv)v \partial_{2n} u }{ (1+|\na{u+sv} (u+sv)|^2)^{3/2}}\bigg).
\end{aligned}
\end{equation}
When we evaluate this expression at $s=0$ to obtain the following statement:  
\[
\begin{aligned}
\frac{d}{ds}\bigg|_{s=0}H_{u+sv} &= \somma{i} X_i^{u}\bigg(\frac{X_i^{u} (v)}{ \sqrt{1+|\na{u} (u)|^2}}\bigg)-\somma{i,j}X_i^{u}\bigg(\frac{X_i^{u} (u)  X_j^{u} (u) X^u_j v }{ (1+|\na{u} (u)|^2)^{3/2}}\bigg)\\
&+v\partial_{2n}\bigg(\frac{X_{2n-1}^{u} (u)}{ \sqrt{1+|\na{u} (u)|^2}}\bigg)+   X_{2n-1}^{u}\bigg(\frac{ v\partial_{2n}(u)}{ \sqrt{1+|\na{u} (u)|^2}}\bigg)\\
& - \somma{i}  X_i^{u} \bigg(\frac{X_i^{u} (u)  X_{2n-1}^{u} (u)v \partial_{2n} u }{ (1+|\na{u} (u)|^2)^{3/2}}\bigg)
\end{aligned}\]

Observing that
\[
\begin{aligned}
\frac{d}{ds}\bigg|_{s=0}H_{u+sv} 
&= \somma{i,j} a_{ij}(\na{u}u)\,X_i^u X_j^u v+\somma{i} \bigg( \somma{j} X_j^u(a_{ij}(\na{u}u)) +\partial_{2n}u \, a_{i 2n-1}(\na{u}u)  \bigg)\, X_i^u v \\
&+ \bigg(  \partial_{2n}\bigg( \frac{ X^u_{2n-1}u }{\sqrt{1+|\na{u}u|^2}} \bigg) +\somma{i} X_i^{u}(\partial_{2n}u \, a_{i 2n-1}(\na{u}u)) \bigg)\, v\\
&=\somma{i,j}X_i^{u} (a_{ij}(\na{u}u)\,X_j^u v)+\somma{i} \partial_{2n}u \, a_{i 2n-1}(\na{u}u) \, X_i^u v \\
&+ \bigg(  \partial_{2n}\bigg( \frac{ X^u_{2n-1}u }{\sqrt{1+|\na{u}u|^2}} \bigg) +\somma{i} X_i^{u}(\partial_{2n}u \, a_{i 2n-1}(\na{u}u)) \bigg)\, v
\end{aligned}
\]
we obtain the thesis.
\end{proof}

We can now introduce the following definition

\begin{definition} We call \emph{stability operator} $L_u$ associated to a minimal graph $G_u$ 
the operator $L_u v$ defined in equation \eqref{def:Lu}. 
\end{definition}

We consider a Euclidean Lipschitz minimal graph $G_u$ with vanishing mean curvature in the Heisenberg group $\hh^n$ and a smooth domain $\Omega$. We say that $G_u$ is \emph{strictly stable} if for all $v \in C^2_h(\Omega)$ with compact support, $v\not\equiv 0$, the \emph{index form}
\[
\mathcal{I}(v,v)=-\int\limits_\Omega v\,L_u v \,d \mc{L}^{2n}
\]
is strictly positive, where $L_u$ is the stability operator defined in \eqref{def:Lu}.

\begin{lemma}\label{small_stable}
Let us note that there exists a constant $C$ such that, if $u$ is a minimum and $||u||_{C^2}\leq C$, 
then $u$ is strictly stable
\end{lemma}\begin{proof}
\[
\begin{aligned} \int v L_u v &=\somma{i,j} \int v X_i^{u} \bigg(\frac{a_{ij}(\na{u}u)}{\sqrt{1+|\na{u}u|^2}}\,X_j^u v\bigg)+\somma{i} \int v \partial_{2n}u \, a_{i 2n-1}(\na{u}u) \, X_i^u v \\ &+ \int v\bigg(  \partial_{2n}\bigg( \frac{ X^u_{2n-1}u }{\sqrt{1+|\na{u}u|^2}} \bigg) +\somma{i} X_i^{u}(\partial_{2n}u \, a_{i 2n-1}(\na{u}u)) \bigg)\, v\\
\end{aligned}\]
Integrating by part the first term, and using the fact that 
$[X_1, X_{n+1}] =\partial_{2n},$, we get
\[
\begin{aligned} \int v L_u v &=- \somma{i,j} \int X_i^{u} v  \frac{a_{ij}(\na{u}u)}{\sqrt{1+|\na{u}u|^2}}\,X_j^u v
-\int  v \partial_{2n}u \frac{a_{ij}(\na{u}u)}{\sqrt{1+|\na{u}u|^2}}\,X_j^u v
\\ &+\somma{i} \int v \partial_{2n}u \, a_{i 2n-1}(\na{u}u) \, X_i^u v \\ &- \int X_{n+1} \bigg( \frac{ X^u_{2n-1}u }{\sqrt{1+|\na{u}u|^2}} \bigg) v X_1 v + \int X_{1} \bigg( \frac{ X^u_{2n-1}u }{\sqrt{1+|\na{u}u|^2}} \bigg) v X_{n+1} v \\ &- \somma{i} \int  \partial_{2n}u \, a_{i 2n-1}(\na{u}u)   X_i^{u} v v - \int  (\partial_{2n}u)^2 \, a_{i 2n-1}(\na{u}u) v^2
\end{aligned}\]
By the structure of the equation we have 
$$\somma{i,j} \int X_i^{u} v  \frac{a_{ij}(\na{u}u)}{\sqrt{1+|\na{u}u|^2}}\,X_j^u v \geq 
\int | X_i^{u} v|^2 \geq \int |v|^2 
$$
By H\"older inequality and Sobolev embedding Theorem we have 
$$\int v X_i^{u} v \leq \int | X_i^{u} v||v|  \leq C_S \int | X_i^{u} v|^2$$
so that 
$$- \int v L_u v \geq (1 - C_S ||u||^2_{C^2})\int | X_i^{u} v|^2.
$$
It follows that $- \int v L_u v$ is strictly positive if $||u||^2_{C^2}$ is sufficiently 
small. 
\end{proof}

\section{An existence result}

In this section we prove a first existence result for solutions of the problem 
\eqref{critical_points}.
Precisely we show that 
\begin{proposition}\label{existence} Let $u\in C^2_h(\tilde{\Om})$ and let $G_u$ be critical point of the area functional in $\hh^n$, $n>1$. We assume that $G_u$  is  strictly stable in a domain $\Omega\subset\tilde{\Om}$. Then there exist $\eps_0>0$, 
such that if $||\phi - u||_{L^\infty(\partial \Omega)} \leq \eps_0$, 
then the problem \eqref{critical_points}
has an unique solution $v\in C^\infty(\Omega)$. 
\end{proposition}

Let us explicitly note that, 
if  $u$ is a solution of class $C^2_h$ of \eqref{critical_points}, then it is of class $C^\infty$. 
Using mollifiers we can mimic the proof of \cite[Theorem~1.2]{MR2774306} and we can easily conclude the following

\begin{lemma}\label{cduetosmooth}
 Let $G_u=\{z=u(x): x\in \Omega \}$ be a $C^{2}_h$  graph in $\hh^n$ with mean curvature $H_u=g\in C^\infty(\Om)$, $n>1$. Then $u$ is a smooth function.  
\end{lemma}
\begin{proof}

We denote by $u_\eps$ the standard mollificators 
\[
u_\eps(x)=\int\limits_{\rr^{2n}} u(y) \varphi_\eps (x-y) \, \eps^{-n}\, dy. 
\]
It is simple to verify that
\begin{itemize}
\item [(i)]  $X_{\bar{u}} f_\eps\rightarrow X_{\bar{u}} f$ uniformly on compact subsets of $\Om$, for $\eps\rightarrow 0$.
\item [(ii)] $X_{\bar{u}}^2 f_\eps\rightarrow X_{\bar{u}}^2 f$ uniformly on compact subsets of $\Om$, for $\eps\rightarrow 0$.
\end{itemize}
For every $\epsilon$, the function $u_\eps$ satisfies the representation formula 
(4.2) in \cite{MR2774306}. 
Letting $\eps$ go to $0$, the same formula is satisfied by $u$
Hence we can now proceed as in \cite[Section~4]{MR2774306}, and obtain the 
smoothness result.   

\end{proof}

Finding a solution of the problem \eqref{critical_points} is equivalent to prove the 
invertibility of the map 
$$F:C^{2,\alpha}_d(\Omega)\rightarrow  C^{0,\alpha}_d(\Omega)\times C(\partial\Omega)$$ 
defined by 
\[
F(w)=(H_w,w\big|_{\partial\Omega}),
\] 

As it is well known, the local invertibity property of $F$ can be studied through its differential $dF(u)=L_u$, 
so that we focus on this linear operator. The continuiuty of its inverse is expressed by the Schauder 
estimates at the boundary, in suitable $C^\alpha$ spaces:

\begin{proposition}\label{prop:scauder} Let $L_u$ defined by \eqref{def:Lu}, where $u:\Omega\subset \rightarrow \rr$ is a smooth function. Let $f\in C^\alpha_d(\Omega)$ and let $v\in C^{2+\alpha}_{loc}(\Omega)$ a bounded function satisfying $L_u v=f$ in $\Omega$. Then $v\in C^{2+\alpha}_{d}(\Omega)$ and there exists $C>0$ (independent of $v$) such that
\[
|| v||_{C_d^{0,\alpha}( \Omega)}\leq C \left( ||v||_{L^\infty(\Omega)}  + ||d^2 f||_{C_d^{0,\alpha}( \Omega)} \right). 
\]
\end{proposition}

The proof of Proposition \ref{prop:scauder} is standard with the techniques developed in the last two decades, since $L_u$ is a sub-elliptic second order linear operator with smooth coefficients. Interior Schauder estimates for $L_u$ can be found in \cite{MR2298970}. Here we need similar estimates, but we need to provide an explicit 
estimate of the constant $C$ in terms of the distance from $K$ to the exterior of $\Omega$. 

Due to the existence of a fundamental solution for $L_u$, it is also possible 
to obtain these estimates, mimicking the classical argument 
presented in the Euclidean setting in \cite[\S~4]{MR1814364}.

%IO SUGGERISCO DI FARE UN PO?DI CONTI

\begin{proof}[Proof of Lemma \ref{existence}] We consider the map $F:C^{2,\alpha}_d(\Omega)\rightarrow  C^{0,\alpha}_d(\Omega)\times C(\partial\Omega)$ defined by 
\[
F(w)=(H_w,w\big|_{\partial\Omega}),
\] 
where $H_w$ is the mean curvature of the graph $G_{w}$. 
The differential of $F$ in $u$ is 
\[
dF_u(v)=(L_u v,v\big|_{\partial\Omega}).
\]
The kernel of $dF_u$ does not contain a non-trivial function, since otherwise the first eigenvalue of $L_u$ in $\Omega$ would be smaller than or equal to $0$. On the other hand, the problem 
\[
\begin{cases}
L_u v=f, & \text{in } \Omega\\
v=\phi, & \text{in } \partial\Omega
\end{cases}
\]
has a solution for all $f\in C^{0,\alpha}_d(\Omega)$ and $\phi\in C(\Omega)$.
From Proposition \ref{prop:scauder} 
the inverse of $dF_u$ is continuous. By the Implicit Function Theorem, there is a diffeomorphism from a neighborhood of $u$ in $C^{2,\alpha}_d(\Omega)$ into a neighborhood of $(H_u, u_{|\partial \Omega})$ in $C^{0,\alpha}_d(\Omega)\times C(\partial\Omega)$. In particular, there exists $\eps_0>0$ so that, for every $\phi$ such that 
$||\phi - u||_{L^\infty(\partial \Omega)} \leq \eps_0$, 
the given problem has a solution 
 $v$ such that the graph $G_{v}$ of $v$ is area-stationary with zero mean curvature and $G_{v}\big|_{\partial\Omega}=\psi$.

Finally the function $v$ is of class $C^\infty$ for Corollary \ref{cduetosmooth}. 
\end{proof}

\section{Area minizing property}

Let us note that, if $u$ is strictly stable, ten the first eigenvalue $\lambda_1(\Omega)$ of the operator $L_u$ is positive, in fact if $\lambda$ is an eigenvalue of $L_u$ with eigenfunction $v\in C^2_{h}(\Omega)\cap C_{0}(\Omega)$ we have
\[
0<-\int\limits_\Omega v\,L_u v  \,d \mc{L}^{2n}=\lambda \int\limits_{\Omega}  v^2 \,d \mc{L}^{2n}.
\]
As a consequence the following Maximum principle holds:

\begin{lemma}\label{lem:maximum principle}
We consider the operator $L_u v$ defined in \eqref{def:Lu}. Suppose that $\lambda_1(\Omega)>0$, where $\lambda_1(\Omega)$ denotes the first eigenvalue of $L$ on a bounded $C^{2,\alpha}$ domain $\Omega\subset \rr^n$. Assume that $Lv\leq 0$ and $\inf_{\partial\Omega}v>0$, then $\inf_{\Omega}v>0$.
\end{lemma}

\begin{proof} First we observe that the minimum of $v$ can not be zero, otherwise $v\equiv 0$. This can be achieved mimic the classical proof of the Hopf's maximum principle (see for example \cite[Theorem~3.5]{MR1814364} or \cite[\S~3]{MR2639314}) and replacing the role of Euclidean balls with the sets $\delta_{r}O_1$, $r>0$, introduced in \cite[p.~1161]{MR1919700}.

Now we suppose that $v$ achieved a negative minimum in the interior of $\Omega$, then we can extend the the operator $L_u$ in a small neighborhood $\tilde{\Omega}$ of $\Omega$ and we can suppose that the first eigenvalue $\lambda_1(\tilde{\Omega})$ is positive as well as the restriction to $\Omega$ of the eigenfunction $v_1$ associated to $\lambda_1(\tilde{\Omega})$. We consider the function $w=v/v_1$ on $\Omega$, that must has a negative minimum in some interior point $p$. Hence at the point $p$
\[
\begin{aligned}
0\geq L_u(v)=L_u(w v_1)=L_u(v_1)w+v_1 \somma{i,j}X_i^{u} (a_{ij}(\na{u}u)\,X_j^u w)\geq -\lambda_1(\tilde{\Omega}) v_1 w
\end{aligned}
\]
and we can conclude that $w(p)\geq 0$ and $v(p)\geq 0$ in contradiction with our assumption. 
\end{proof}

\begin{lemma}\label{lem:existencecalibration} Let $u\in Lip(\tilde{\Om}) \cap C^2_h(\Om)$ and let $G_u$ be a  minimal graph in $\hh^n$, $n>1$. We assume that $G_u$  is  strictly stable in a domain $\Omega\subset\tilde{\Om}$. Then there exist $\eps_0>0$, a tubular neighborhood $U$ of $G_u$ and a family $G_{u_\eps}$ of  surfaces with vanishing mean curvature such that $G_{u_0}=G_u$ and  $G_{u_\eps}$ is a foliation of $U$ for $\eps\in [-\eps_0,\eps_0]$.
\end{lemma}

\begin{proof} As before we consider the map 
\[
F(w)=(H_w,w\big|_{\partial\Omega}). 
\] 
We have proved that there is a diffeomorphism from a neighborhood of 
$u$ in $C^{2,\alpha}_d(\Omega)$ into a neighborhood of $(H_u, u_{|\partial \Omega})$ in $C^{0,\alpha}_d(\Omega)\times C(\partial\Omega)$. In particular, there exists $\eps_0>0$ so that, for all $\eps\in(-\eps_0,\eps_0)$, there is a function $u_\eps$ such that the graph $G_{u_\eps}$ of $u_\eps$ is area-stationary with zero mean curvature and $G_{u_\eps}\big|_{\partial\Omega}=u\big|_{\partial\Omega} + \eps$.

Let us check that the union of the graphs $G_{u_\eps}$ provide a foliation of a tubular neighborhood of $G_u$ in $\hh^n$. The graphs $G_{u_\eps}$ provide a variation of $G_{u}$ in $\hh^n$. If we
compute the variational function $v:=\frac{d}{d\eps}\big|_{\eps=0}u_\eps$, then $v$ satisfies 

\begin{displaymath}
\begin{cases}
L_u v=0,    & \text{ in } \Omega\\
v=1,      & \text{ in } \partial \Omega
\end{cases}.
\end{displaymath}
Since by Lemma \ref{lem:maximum principle} below $v>0$, choosing $\eps_0$ smaller if necessary, we would obtain that the graph $G_{u_\eps}$, for $\eps\in(-\eps_0,\eps_0)$ foliate a tubular neighborhood of $G_u$.

\end{proof}

Now we are ready to prove the main result of the paper

\begin{theorem}\label{thm:calibration} Let $u\in Lip(\tilde{\Om})\cap C^2_h(\Om)$ and let $G_u$ be a  minimal graph in $\hh^n$, $n>1$. We assume that $G_u$  is  strictly stable in a domain $\Omega\subset\tilde{\Om}$.. Then there exists a tubular neighborhood $U$ of $\Omega$ such that for any  $C^1_h$ graph $S\subset U$, $\partial\Omega=\partial S$, we have $\area(G_u)\leq \area(S)$ or $G_u=S$.
\end{theorem}

\begin{proof} From Lemma \ref{lem:existencecalibration} we know the existence of a family $\{G_{u_\eps}\}_{\eps\in(-\eps_0,\eps_0)}$ of minimal graphs $G_{u_\eps}$ that foliate a tubular neighborhood $U$ of $G_u$. Let $v:\Om\rightarrow \rr$ be a $C^1_h(\Om)$ function such that $u\equiv v$ on $ \partial\Om$. We denote by $E$ the region bounded by $G_u \cup G_v$. 
Using the Gauss-Green formula, \cite[Theorem~2.3]{MR2979606}, and the fact that $\escpr{(\nu_h)_{G_{u_\eps}},(\nu_h)_{G_u}}=1$ and $\escpr{(\nu_h)_{G_{u_\eps}},(\nu_h)_{G_v}    }\geq -1$, we get
\[
\begin{split}
0=\int\limits_E \divv_{\hh^n}((\nu_h)_{\Sg_\eps})\, d\mc{L}^{2n+1}&= \int\limits_{\hh^n}\escpr{(\nu_h)_{G_{u_\eps}},(\nu_h)_{G_u}} \, dP(G_u)+\int\limits_{\hh^n}\escpr{(\nu_h)_{G_{u_\eps}},(\nu_h)_{G_v}    } \, dP(G_v)\\
&\geq \area(G_u\cap U)-\area(G_v\cap U),
\end{split}
\]
with strict inequality unless $(\nu_h)_{G_v}$ coincides with $(\nu_h)_{G_u}$ in all points. Here $(\nu_h)_{G_u}$ denotes the horizontal projection of the Riemannian unit normal of $G_u$. We conclude that $G_u$ and $G_v$ coincides by \cite[Corollary~1.3]{MR2600502}.
\end{proof}

As a direct consequence of this result and the fact that $0$ is a stable solution, 
we obtain 
\begin{corollary}
\label{existence_minima}
Then there exist $\eps_0>0$, 
such that if $||\phi||_{L^\infty(\partial \Omega)} \leq \eps_0$, 
then the problem 
\[
\begin{cases}
H_v =0, & \text{in } \Omega\\
v=\phi, & \text{in } \partial\Omega
\end{cases}
\]
has an unique solution $v\in C^\infty(\Omega) \cap C(\bar \Omega)$, which is a stable minimum for 
the area functional. 

\end{corollary}

\bibliography{calibrations01082016}

\end{document}